\documentclass[final]{siamltex}
\usepackage{a4}
\usepackage{geometry} 
\usepackage{amsmath}
\usepackage{amssymb}
\usepackage{graphicx}
\usepackage{caption,subcaption}
\usepackage{multirow}
\usepackage{xstring}
\usepackage[utf8]{inputenc}
\usepackage{hyperref}
\usepackage{amstext,amsbsy,amsopn,eucal,enumerate}
\usepackage{tikz}
\usepackage{pgfplots}
\usepackage{tabularx}

\usepackage{soul}
\usepackage{algorithm} 
\usepackage{algpseudocode}

\def\R{\mathbb{R}}
\def\C{\mathbb{C}}

\newcommand{\expn}{\operatorname{e}}

\newcommand{\orth}{\operatorname{orth}}
\newcommand{\vect}{\operatorname{vec}}

\newcommand{\beq}{\begin{equation}}
\newcommand{\eeq}{\end{equation}}
\newcommand {\mat}      [1] {\left[\begin{array}{#1}}
\newcommand {\rix}          {\end{array}\right]}
\newcommand {\smat}      [1] {\left[\begin{smallmatrix}{#1}}
\newcommand {\srix}          {\end{smallmatrix}\right]}
\newcommand {\s}      [1] {\begin{smallmatrix}{#1}}
\newcommand {\se}          {\end{smallmatrix}}
\newcommand{\trace}{\operatorname{tr}}

\newcommand{\hA}{\ensuremath{\hat{A}}}
\newcommand{\hB}{\ensuremath{\hat{B}}}
\newcommand{\hC}{\ensuremath{\hat{C}}}
\newcommand{\hN}{\ensuremath{\hat{N}}}

\newcommand{\tB}{\ensuremath{\tilde{B}}}
\newcommand{\tC}{\ensuremath{\tilde{C}}}
\newcommand{\tN}{\ensuremath{\tilde{N}}}

\newtheorem{defn}{Definition}[section]
\newtheorem{remark}{Remark}

\newtheorem{lem}[defn]{Lemma}
\newtheorem{kor}[defn]{Corollary}
\newtheorem{thm}[defn]{Theorem}

\usetikzlibrary{external}
\def\addlegendimage{\csname pgfplots@addlegendimage\endcsname}

\newlength\fheight
\newlength\fwidth

\title{Bilinear systems -- A new link to $\mathcal H_2$-norms, relations to stochastic systems and further properties}

\author{Martin Redmann\thanks{Martin Luther University Halle-Wittenberg, Institute of Mathematics, Theodor-Lieser-Str. 5, 06120 Halle (Saale), Germany, Email: {\tt 
martin.redmann@mathematik.uni-halle.de}.  Financial support by the DFG via Research Unit FOR 2402 is gratefully acknowledged.}}

\begin{document}

\maketitle

\begin{abstract}
In this paper, we prove several new results that give new insights into bilinear systems. We discuss conditions for asymptotic stability using probabilistic arguments. 
Moreover, we provide a global characterization of reachability in bilinear systems based on a certain Gramian. Reachability energy estimates using the same Gramian have only been local so far.
The main result of this paper, however, is a new link between the output error and the $\mathcal H_2$-error of two bilinear systems. This result has several consequences in the field of model order reduction. It explains why 
$\mathcal H_2$-optimal model order reduction leads to good approximations in terms of the output error. Moreover, output errors based on the $\mathcal H_2$-norm can now be proved for balancing related model
order reduction schemes in this paper. All these new results are based on a Gronwall lemma for matrix differential equations that is established here.
\end{abstract}



\pagestyle{myheadings}
\thispagestyle{plain}
\markboth{M. REDMANN}{BILINEAR SYSTEMS -- LINKS TO $\mathcal H_2$-NORMS AND STOCHASTIC SYSTEMS AND MORE}


\section{Introduction}

In this paper, we consider bilinear control systems that have applications in various fields \cite{brunietal,Mohler,rugh1981nonlinear}. These systems are of the form
\begin{equation}\label{sys:original}
\Sigma:\left\{ \begin{aligned}
\dot x(t) &= Ax(t) +  Bu(t) + \sum_{k=1}^m N_k x(t) u_k(t),\quad x(0) = x_0,\\
y(t) &= Cx(t),\quad t\geq 0,
\end{aligned}\right.
\end{equation}
where $A, N_k \in \R^{n\times n}$, $B \in \R^{n\times m}$, and $C \in \R^{p\times n}$ are constant matrices. The vectors $x$, $u=(u_1\ u_2\, \ldots\, u_m)^T$ and $y$ 
denote the state, the control input and the quantity of interest (output vector), respectively. The solution of the state equation in (\ref{sys:original}) is denoted by $x(t, x_0, B)$ to indicate the dependence
on the initial condition $x_0$ and the input matrix $B$. The solution $x(t, x_0, 0)$ to the homogeneous state equation is sometimes considered with an initial time $s$ different from zero and is then written 
as $x_{(s)}(t, x_0, 0)$. We also assume that the matrix $A$ is Hurwitz, meaning 
that $\sigma(A) \subset \C_{-}=\{z\in\mathbb C:\quad \Re(z)<0\}$, where $\sigma(\cdot)$ denotes the spectrum of a matrix and $\Re(\cdot)$ represents the real part of a complex number. Furthermore, let 
$u\in L^2$, i.e., \begin{align*}
                   \left\|u\right\|_{L^2}^2:= \int_0^\infty \left\|u(s)\right\|^2_2 ds = \int_0^\infty u^T(s) u(s) ds <\infty.
                  \end{align*}
System (\ref{sys:original}) 
can, e.g., represent a spatially discretized partial differential equation. Then, $n$ is usually large and solving (\ref{sys:original}) becomes computationally expensive, in particular if the system has to be evaluated for many 
controls $u$. Therefore, model order reduction (MOR) is vital aiming to replace the original large scale system by a system of small order in order to reduce complexity. We introduce such a reduced order model (ROM) as follows:
\begin{equation}\label{sys:reduced}
\hat\Sigma:\left\{ \begin{aligned}
\dot {\hat x}(t) &= \hat A \hat x(t) +  \hat Bu(t) + \sum_{k=1}^m \hat N_k \hat x(t) u_k(t), \quad \hat x(0) = \hat x_0,\\
\hat y(t) &= \hat C\hat x(t),\quad t\geq 0,
\end{aligned}\right.
\end{equation}
where  $\hat A, \hat N_k\in \R^{r\times r}$, $\hat B \in \R^{r\times m}$, and $\hat C \in \R^{p\times r}$ with $r \ll n$. In order to determine the quality of the reduction, it is essential to find a bound $\mathcal E\geq 0$
which estimates the error between $y$ and $\hat y$, e.g., as follows 
\begin{align}\label{abstractbound}
 \sup_{t\geq 0}\left\|y(t)-\hat y(t)\right\|_2 \leq \mathcal E f(u)
\end{align}
assuming zero initial conditions, where $f$ is a suitable function. \smallskip

In this paper, we first find an estimate for the homogeneous state variable in (\ref{sys:original}), which is based on a new Gronwall lemma for matrix differential equations. In the upper bound of this estimate, 
the dependence on the control $u$ is decoupled allowing to analyze several properties of system (\ref{sys:original}). We are hence able to show a link between deterministic bilinear systems and linear stochastic systems 
which means that results from the stochastic case can be transferred to bilinear systems. As a consequence, we show asymptotic stability for bilinear systems under the above assumptions on $u$ and $A$ using probabilistic arguments.\smallskip

As the main result of this paper, we will moreover prove that there is an $f$ such that $\mathcal E = \left\|\Sigma-\hat \Sigma\right\|_{\mathcal H_2}$
in (\ref{abstractbound}), i.e., the output error between (\ref{sys:original}) and (\ref{sys:reduced}) can be bounded by an expression depending on the $\mathcal H_2$-error between both bilinear systems.
This connection has been an open problem for a long time. This is an important new finding since it finally answers 
the question why $\mathcal H_2$-optimal MOR techniques like the bilinear iterative rational Krylov algorithm \cite{morBenB12b, flagg2015multipoint} lead to good approximations. Moreover, it is possible to
find output error bounds based on the $\mathcal H_2$-error for balancing related MOR schemes 
like balanced truncation (BT) \cite{typeIBT,morBenD11} and singular perturbation approximation (SPA) \cite{hartmann}, which could not be established so far. We will provide these bounds
that are based on the analysis of $\mathcal E$ in the context of MOR for stochastic equations \cite{redmannbenner,redmannspa2,BTtyp2EB,redSPA} again showing the connection between bilinear and stochastic systems.
The new error bounds then allow us to point out the situations in which balancing related methods perform well for bilinear systems. \smallskip 

In the case of BT, a different kind of bound has already been 
found in \cite{beckerhartmann}. Using a different reachability Gramian in comparison to the work in \cite{typeIBT,morBenD11, hartmann}, error bounds for BT and SPA could be achieved \cite{redstochbil, redmannstochbilspa}. 
We will also discuss reachability in bilinear systems. Using the output error bound established in this paper, an estimate based on the 
reachability Gramian proposed in \cite{typeIBT} is derived. This estimate allows us to identify states in the system dynamics that require a larger amount of energy to be reached. Previous 
characterizations of reachability based on the same Gramian have only been local so far \cite{morBenD11, enefungray98}.

\section{Fundamental solutions, solution representations and Gronwall lemma}

We briefly recall the concept of fundamental solutions to deterministic bilinear systems. Furthermore, we state their solution representations. Those will be essential for the error analysis 
for bilinear systems. We introduce the fundamental solution $\Phi_u(t,s)$, $s\leq  t$, of the state equation in (\ref{sys:original}) as a matrix-valued function solving
\begin{align}\label{fun_sol_bil}
 \Phi_u(t,s) = I +\int_s^t A \Phi_u(\tau,s) d\tau + \sum_{k=1}^m \int_s^t N_k \Phi_u(\tau,s) u_k(\tau) d\tau.
\end{align}
For initial time $s=0$, we simply write $\Phi_u(t):=\Phi_u(t,0)$. In order to distinguish between homogeneous state variables ($B=0$) with different initial times, we introduce $x_{(s)}(\cdot, x_0, 0)$ 
as the homogeneous solution to the state equation (\ref{sys:original}) with initial time $s\geq 0$. The subscript is omitted if $s=0$. By multiplying (\ref{fun_sol_bil}) with $x_0$ from the right, it can be seen that 
the solution to the homogeneous system with initial time $s\geq 0$ is
\begin{align*}
x_{(s)}(t,x_0, 0) &= \Phi_u(t,s) x_0, \quad x_{(s)}(s, x_0, 0)=x_0.
\end{align*}
The fundamental solution moreover allows us to find an explicit expression for the solution to the bilinear state equation in (\ref{sys:original}) for general $B$. It is given in 
the next theorem.
\begin{thm}\label{sol_rep}
Let $x(t, x_0, B)$, $t\geq 0$, be the solution to the state equation in (\ref{sys:original}). Then, it has the following representation:
\begin{align*}
 x(t, x_0, B) = \Phi_u(t) x_0 + \int_0^t \Phi_u(t,s) B u(s) ds,
\end{align*}
where $\Phi_u$ is fundamental solution to the bilinear system.
\end{thm}
\begin{proof}
The identity for $x$ is obtained by applying the product rule to $\Phi_u(t) g(t)$, where $g(t) := x_0 + \int_0^t \Phi_u^{-1}(s) B u(s) ds$ exploiting that $\Phi_u(t,s) = \Phi_u(t) \Phi_u^{-1}(s)$. 
\end{proof}\\
Since $\Phi_u$ depends on the control $u$, the error bound analysis for bilinear systems is significantly harder than the one for deterministic linear systems ($N_k=0$), where the fundamental solution 
becomes $\Phi_u(t, s) = \expn^{A(t-s)}$.
That is why we will prove a Gronwall type result which allows to find an upper bound for a quadratic form of the fundamental solution, in which the dependence of $u$ is decoupled. Such a 
Gronwall lemma for matrix differential equations will be established in the following.\smallskip

Given two matrices $K$ and $L$, we write $K \leq L$ below if $L-K$ is a symmetric positive semidefinite matrix. Moreover, we introduce the vector of control components with a non-zero $N_k$, $k\in\{1, \ldots, m\}$: 
\begin{align}\label{uzero}
u^{0}=(u_1^{0}\ u_2^{0}\, \ldots\, u_m^{0})^T\quad \text{with}\quad u_k^{0} \equiv \begin{cases}
  0,  & \text{if }N_k = 0\\
  u_k, & \text{else}.
\end{cases}
\end{align}
We provide three vital lemmas, which form the basis for the results in this paper. Their proofs are moved to the appendix in order to improve the readability of the paper.
\begin{lem}\label{lem1}
Let $x_{(s)}(t, x_0, 0)$, $t\geq s\geq 0$, denote the solution to the homogeneous bilinear equation \begin{align}\label{hom_bil_eq}
\dot x(t) &= Ax(t) + \sum_{k=1}^m N_k x(t) u_k(t),\quad x(s)=x_0.
\end{align}
Then, the function $x_{(s)}(t, x_0, 0) x_{(s)}^T(t, x_0, 0)$, $t\geq s\geq 0$, satisfies the following matrix differential inequality:
\begin{align}\label{matrix_ineq}
 \dot X(t) \leq A X(t) + X(t) A^T +\sum_{k=1}^m N_k X(t) N_k^T + X(t) \left\|u^{0}(t)\right\|_2^2,
\end{align}
where $X(s)=x_0 x_0^T$.
\end{lem}
\begin{proof}
See Appendix \ref{secprooflem1}.
\end{proof}\\
We now find a representation for the respective equality.
\begin{lem}\label{lem2}
Let $\bar Z(t, Z_0)$, $t\geq 0$, satisfy 
\begin{align}\label{matrix_eq_sinu}
 \dot {\bar Z}(t) = A \bar Z(t) + \bar Z(t) A^T +\sum_{k=1}^m N_k \bar Z(t) N_k^T ,\quad \bar Z(0) = Z_0\geq 0.
\end{align}
Then, the function $\exp\left\{\int_s^t \left\|u^{0}(s)\right\|_2^2 ds\right\} \bar Z(t-s, Z_0)$, $t\geq s\geq 0$, solves the following matrix differential equation:
\begin{align}\label{matrix_eq}
 \dot Z(t) = A Z(t) + Z(t) A^T +\sum_{k=1}^m N_k Z(t) N_k^T + Z(t) \left\|u^{0}(t)\right\|_2^2,
\end{align}
where $Z(s)=Z_0$.
\end{lem}
\begin{proof}
The result follows by applying the product rule to $Z(t)=\exp\left\{\int_s^t \left\|u^{0}(s)\right\|_2^2 ds\right\} \bar Z(t-s, Z_0)$.
\end{proof}\\
A Gronwall lemma follows next which yields a relation between the above matrix (in)equalities.
\begin{lem}\label{lem3}
Let the matrix-valued function $X(t)$, $t\geq s\geq 0$, satisfy (\ref{matrix_ineq}) and let $Z(t)$, $t\geq s\geq 0$, be the solution to the matrix differential equation (\ref{matrix_eq}).
If $X(s)\leq Z(s)$, we have that $X(t)\leq Z(t)$ for all $t\geq s\geq 0$.
\end{lem}
\begin{proof}
See Appendix \ref{secprooflem3}.
\end{proof}\\
From Lemmas \ref{lem1}, \ref{lem2} and \ref{lem3} it immediately follows that 
\begin{align}\label{essentialest}
x_{(s)}(t, x_0, 0) x_{(s)}^T(t, x_0, 0)\leq \exp\left\{\int_s^t \left\|u^{0}(\tau)\right\|_2^2 d\tau\right\}\bar Z(t-s, x_0 x_0^T)
\end{align}
for all $t\geq s\geq 0$ by setting $Z_0 = x_0 x_0^T$. Inequality (\ref{essentialest}) implies an estimate for $\Phi_u$ that we derive in Theorem \ref{mainthm} to establish an output bound for bilinear systems. 
Before we get to this bound, we briefly discuss the relation between linear stochastic and deterministic bilinear systems.

\section{Relation to stochastic systems and their stability}

Let us introduce the following stochastic differential equation \begin{align}\label{hom_stoch_eq}
d z(t) = Az(t) dt + \sum_{k=1}^m N_k z(t) dw_k(t),\quad z(0)=x_0,
\end{align}
where $w_1, \ldots, w_m$ are independent standard Brownian motions. It is obtained by replacing the control components in (\ref{hom_bil_eq}) by the ``derivatives'' of these stochastic processes. Since these 
derivatives are no longer classical functions, there is a completely different theory behind (\ref{hom_stoch_eq}) in comparison to the respective homogeneous bilinear system. Still (\ref{essentialest}) provides 
an interesting link between both systems which we obtain by a stochastic representation  of the solution $\bar Z$ of (\ref{matrix_eq_sinu}). 
\begin{proposition}\label{prop_stoch_rep}
Let $z(\cdot, x_0)$ denote the solution to (\ref{hom_stoch_eq}) with initial condition $x_0$, then the function $\mathbb E\left[z(t, x_0)z^T(t, x_0)\right]$, $t\geq 0$, solves (\ref{matrix_eq_sinu}) with $Z_0 = x_0 x_0^T$.
\end{proposition}
\begin{proof}
 We refer to, e.g., \cite{damm} or \cite{redmannspa2} for a proof.
\end{proof}\\
Consequently, (\ref{essentialest}) with $s=0$ becomes \begin{align}\label{essentialest_stoch}
x(t, x_0, 0) x^T(t, x_0, 0)\leq \exp\left\{\int_0^t \left\|u^{0}(\tau)\right\|_2^2 d\tau\right\} \mathbb E\left[z(t, x_0)z^T(t, x_0)\right]
\end{align}
using Proposition \ref{prop_stoch_rep}. It means that one can read off properties of (\ref{hom_bil_eq}) from properties of (\ref{hom_stoch_eq}) such as stability.

We discuss mean square asymptotic stability of (\ref{hom_stoch_eq}) and its consequences for (\ref{hom_bil_eq}) in the following because this condition will 
appear as an assumption in the main theorem below.
\begin{lem}\label{lemstab}
The following two statements are equivalent:
\begin{itemize}
 \item Equation (\ref{hom_stoch_eq}) is exponentially mean square stable, that is, there exist $k_1, k_2>0$, such that 
\begin{align}\label{expmanstab0}
\mathbb E \left\|z(t, x_0)\right\|^2_2\leq \left\|x_0\right\|^2_2 k_1 \expn^{-k_2 t}.
\end{align}
\item The eigenvalues of $A\otimes I+I\otimes A+ \sum_{k=1}^m N_k \otimes N_k$ have negative real parts only, i.e.,
 \begin{align}\label{expmanstab}
 \sigma(A\otimes I+I\otimes A+ \sum_{k=1}^m N_k \otimes N_k)\subset \mathbb C_-.  
\end{align}
\end{itemize}
Moreover, $\sigma(A)\subset \mathbb C_-$ and \begin{align}\label{suffsstab}
                \left\|\int_0^\infty \expn^{A t} \left(\sum_{k=1}^m N_k N_k^T\right)\expn^{A^T t}dt\right\|_2<1        
                       \end{align}
imply (\ref{expmanstab}).
\end{lem}
\begin{proof}
The equivalence of (\ref{expmanstab0}) and (\ref{expmanstab}) is a well-known result for stochastic systems. A proof can, e.g., be found in \cite{damm, staboriginal, redmannspa2}. That 
$\sigma(A)\subset \mathbb C_-$ and (\ref{suffsstab}) are sufficient for exponentially mean square stability is proved in \cite{damm, haus, wonham}.
\end{proof}\\
Applying the trace operator to both sides of (\ref{essentialest_stoch}) and using Lemma \ref{lemstab}, condition (\ref{expmanstab}) yields \begin{align*}
     \left\|x(t, x_0, 0)\right\|^2_2 \leq \exp\left\{\int_0^t \left\|u^{0}(\tau)\right\|_2^2 d\tau\right\} \left\|x_0\right\|^2_2 k_1 \expn^{-k_2 t}                                                                                                      
                                                                                                             \end{align*}
and hence asymptotic stability of the homogeneous bilinear equation if $u\in L^2$. In Section \ref{sec_main_res}, we will see that even $\sigma(A)\subset \mathbb C_-$ is sufficient for asymptotic stability in a homogeneous bilinear 
system.

\section{An $\mathcal{H}_2$-bound for bilinear systems and further properties}\label{sec_main_res}

In this section, we are now able to establish several new results for system (\ref{sys:original}) based on Lemmas \ref{lem1}, \ref{lem2}, \ref{lem3} and their consequences discussed above. 
We begin with the main result of this paper which is an output bound leading to an output error bound between systems (\ref{sys:original}) and (\ref{sys:reduced}). Subsequently, we discuss the identification of less relevant 
states in a bilinear system based on a new estimate. Moreover, we show that the output of a bilinear systems is bounded and the homogeneous state equation is asymptotically stable given that the matrix $A$ is Hurwitz.

\subsection{Output bounds and reachability estimate}

The next theorem provides an output bound depending on the $\mathcal H_2$-norm of a bilinear system.
\begin{thm}\label{mainthm}
Let $y$ be the output of system (\ref{sys:original}) with $x_0=0$ and suppose that
 \begin{align}\label{assumption}
 \sigma(A\otimes I+I\otimes A+ \sum_{k=1}^m N_k \otimes N_k)\subset \mathbb C_-.  
\end{align}
Then, it holds that \begin{align*}
  \sup_{t\geq 0}\left\|y(t)\right\|_2 \leq \left(\trace(C P C^T)\right)^{\frac{1}{2}} \exp\left\{0.5\left\|u^{0}\right\|_{L^2}^2\right\} \left\|u\right\|_{L^2},                   
                    \end{align*}
where $P$ is the unique solution to \begin{align}\label{lyapeq}
                  A P + P A^T + \sum_{k=1}^m N_k P N_k^T = - B B^T.
                                                                                              \end{align}
\end{thm}
\begin{proof}
Let $y(t)= C x(t, 0, B)$ be the output of (\ref{sys:original}) with zero initial state. Then, with the representation from Theorem \ref{sol_rep}, we have  
\begin{align}\label{estimstey}
\left\|y(t)\right\|_2&=\left\|C \int_0^t \Phi_u(t, s) B u(s) ds\right\|_2 \leq \int_0^t \left\|C \Phi_u(t, s) B u(s)\right\|_2 ds\\ \nonumber
&\leq \int_0^t \left\|C \Phi_u(t, s) B\right\|_F \left\|u(s)\right\|_2 ds \leq \left(\int_0^t \left\|C \Phi_u(t, s) B \right\|_F^2 ds\right)^{\frac{1}{2}} \left(\int_0^t \left\|u(s)\right\|_2^2 ds\right)^{\frac{1}{2}}.
\end{align}
We further analyze the term $\int_0^t \left\|C \Phi_u(t, s) B \right\|_F^2 ds$. 
We partition $B = [b_1, b_2, \ldots, b_m]$, where $b_k$ is the $k$th column of $B$. We have \begin{equation}\label{auxident}\begin{aligned}
  \Phi_u(t, s) B = [\Phi_u(t, s) b_1, \Phi_u(t, s) b_2, \ldots, \Phi_u(t, s) b_m] 
  = [x_{(s)}(t, b_1 , 0), x_{(s)}(t, b_2 , 0), \ldots, x_{(s)}(t, b_m , 0)].                                                                                \end{aligned}
                                                                                                                                                                        \end{equation}
 As mentioned before, Lemmas \ref{lem1}, \ref{lem2} and \ref{lem3} provide (\ref{essentialest}). Via (\ref{auxident}) and (\ref{essentialest}), we obtain \begin{align*}
\Phi_u(t, s) B B^T \Phi^T_u(t, s) &= \sum_{k=1}^m x_{(s)}(t, b_k , 0)x_{(s)}^T(t, b_k , 0) \leq  \exp\left\{\int_s^t \left\|u^{0}(\tau)\right\|_2^2 d\tau\right\} \sum_{k=1}^m \bar Z(t-s, b_k b_k^T).
\end{align*}
Since the solution $\bar Z$ to (\ref{matrix_eq_sinu}) is linear in its initial condition and positive semidefinite, we find \begin{align*}
\Phi_u(t, s) B B^T \Phi^T_u(t, s) \leq  \exp\left\{\int_0^t \left\|u^{0}(\tau)\right\|_2^2 d\tau\right\} \bar Z(t-s, BB^T)
\end{align*}
using $\int_s^t \left\|u^{0}(\tau)\right\|_2^2 d\tau\leq\int_0^t \left\|u^{0}(\tau)\right\|_2^2 d\tau$.
This estimated leads to \begin{align*}
  \int_0^t \left\|C \Phi_u(t, s) B \right\|_F^2 ds &=  \int_0^t \trace(C \Phi_u(t, s) B B^T \Phi_u^T(t, s) C^T)ds  \\
   &\leq\exp\left\{\int_0^t \left\|u^{0}(s)\right\|_2^2 ds\right\}  \int_0^t\trace(C \bar Z(t-s, BB^T) C^T)ds \\
   &=\exp\left\{\int_0^t \left\|u^{0}(s)\right\|_2^2 ds\right\} \trace(C\int_0^t \bar Z(s, BB^T) ds\; C^T)
                        \end{align*}
using the linearity of the trace operator and substitution $t-s\mapsto s$. We define $P_t:=\int_0^t \bar Z(s, BB^T)  ds$ and insert the above result into (\ref{estimstey}) which yields 
\begin{align}\label{lastest}
\left\|y(t)\right\|_2\leq \exp\left\{0.5 \int_0^t \left\|u^{0}(s)\right\|_2^2 ds\right\} \left(\trace(C P_t C^T)\right)^{\frac{1}{2}} \left(\int_0^t \left\|u(s)\right\|_2^2 ds\right)^{\frac{1}{2}}.
\end{align}
$P_t$ is increasing in $t$ since $\bar Z$ is positive semidefinite. Moreover, $\left\|\bar Z(t, B B^T)\right\|_2\lesssim \expn^{-k t}$ for some $k>0$ due to assumption (\ref{assumption}). This can be seen by vectorizing (\ref{matrix_eq_sinu}) 
leading to \begin{align*}
\frac{d}{dt} \vect\left(\bar Z(t)\right) =  (A\otimes I+I\otimes A+ \sum_{k=1}^m N_k \otimes N_k)\vect\left(\bar Z(t)\right),
           \end{align*}
such that (\ref{assumption}) implies asymptotic stability of $\vect\left(\bar Z\right)$ and hence the same for $\bar Z$. Therefore, $P=\lim_{t\rightarrow \infty} P_t$ exists. The matrix equation (\ref{lyapeq})
for $P$ is obtained by integrating both sides of (\ref{matrix_eq_sinu}) over $[0, v]$ with $Z_0 = BB^T$ and then taking the limit of $v\rightarrow \infty$. Now, taking the supremum on both sides of (\ref{lastest}), 
the claim follows.
\end{proof}
\begin{remark}\label{remark1} 
It was shown in \cite{morZhaL02} that the term entering the bound in Theorem \ref{mainthm} is the Gramian based representation of the $\mathcal H_2-norm$ of system (\ref{sys:original}), i.e., 
$\left\|\Sigma\right\|_{\mathcal H_2}^2 = \trace(C P C^T)$. When we choose $N_k=0$ for all $k=1, \ldots, m$, then the exponential term in the bound becomes $1$ and hence we obtain the well-known relation 
between the output and the $\mathcal H_2$-norm in the linear case \cite{morGugAB08}.\smallskip

Condition (\ref{assumption}) is needed to guarantee the existence of the Gramian $P$. However, it can be weakened to $\sigma(A)\subset \mathbb C_-$, since the bilinear state equation can be equivalently rewritten as 
\begin{align}\label{rescaled_sys}
 \dot x(t)=Ax(t)+[\frac{1}{\gamma}B][\gamma u(t)]+\sum_{k=1}^m [\frac{1}{\gamma}N_k] x(t) [\gamma u_k(t)],
            \end{align}
see also\cite{morBenD11,morcondon2005}, where the weighted matrices $\tilde N_k=\frac{1}{\gamma}N_k$ can be made arbitrary small with a sufficiently large constant $\gamma>0$. Now, we see that we have
\begin{align*}
           \tilde f(A, \tilde N_k):=  \left\|\int_0^\infty \expn^{A t} \left(\sum_{k=1}^m \tilde N_k \tilde N_k^T\right)\expn^{A^T t}dt\right\|_2 
                \leq \frac{1}{\gamma^2}\sum_{k=1}^m \left\| N_k \right\|_2^2 \int_0^\infty \left\|\expn^{A t}\right\|_2^2  dt.
                       \end{align*}
$\mathcal I:= \int_0^\infty \left\|\expn^{A t}\right\|_2^2  dt$ is finite because $A$ is Hurwitz, such that choosing $\gamma > \sqrt{\sum_{k=1}^m \left\| N_k \right\|_2^2 \mathcal I}$ leads to 
 $\tilde f(A, \tilde N_k)<1$. This, by Lemma \ref{lemstab}, implies  
 \begin{align}\label{scalestab}
 \sigma(A\otimes I+I\otimes A+ \sum_{k=1}^m \tilde N_k \otimes \tilde N_k)\subset \mathbb C_-.  
\end{align}
Then, Theorem \ref{mainthm} can be applied to (\ref{rescaled_sys}) and we get
\begin{align}\label{outbounded}
  \sup_{t\geq 0}\left\|y(t)\right\|_2 \leq \gamma \left(\trace(C P_\gamma C^T)\right)^{\frac{1}{2}} \exp\left\{0.5\gamma^2\left\|u^{0}\right\|_{L^2}^2\right\} \left\|u\right\|_{L^2},                   
                    \end{align}
where $P_\gamma$ solves \begin{align*}
                  A P_\gamma + P_\gamma A^T + \frac{1}{\gamma^2}\sum_{k=1}^m N_k P_\gamma N_k^T = - \frac{1}{\gamma^2}B B^T.
                  \end{align*}
We see that the rescaling makes the bound potentially large for $\gamma\gg 1$ due to the exponential term. However, (\ref{outbounded}) shows that $y$ is always bounded if $u\in L^2$ and $\sigma(A)\subset \mathbb C_-$.
\end{remark}\\
We can derive an inequality from Theorem \ref{mainthm} that can be used to characterize reachability in the bilinear system. It leads to an improved characterization in comparison to 
\cite{morBenD11, enefungray98}, where energy estimates are shown that hold only in a small neighborhood of zero.
\begin{kor}\label{correach}
Let $x(t, 0, B)$, $t\geq 0$, be the solution to the state equation in (\ref{sys:original}) with $x_0=0$ and suppose that 
 (\ref{assumption}) holds. Let $P$ be the solution to (\ref{lyapeq}) and $(v_k)_{k=1, \ldots, n}$ be an orthonormal basis of eigenvectors of $P$ with corresponding eigenvalues $(\lambda_k)_{k=1, \ldots, n}$.
Then, it holds that \begin{align*}
  \sup_{t\geq 0}\left\vert\langle x(t, 0, B), v_k\rangle_2\right\vert \leq \lambda_k^{\frac{1}{2}} \exp\left\{0.5\left\|u^{0}\right\|_{L^2}^2\right\} \left\|u\right\|_{L^2}.          
                    \end{align*}
\end{kor}
\begin{proof}
We set $C=v_k^T$ in Theorem \ref{mainthm}. We then obtain $y(t) = C x(t, 0, B) = \langle x(t, 0, B), v_k\rangle_2$ and $\trace(C P C^T) = \lambda_k$ since the eigenvectors are orthonormal.
\end{proof}\\
Using an orthonormal basis $(v_k)_{k=1, \ldots, n}$ of eigenvectors of $P$, we can write \begin{align*}
x(t, 0, B) = \sum_{k=1}^n \langle x(t, 0, B), v_k \rangle_2\; v_k.
 \end{align*}
If the control $u$ is not too large and if the eigenvalue $\lambda_k$ corresponding to $v_k$ is small, then the Fourier coefficient
$\langle x(\cdot, 0, B), v_{k}  \rangle_2$ is close to zero according to Corollary \ref{correach}. This means that the state variable takes only very small values in the direction of $v_{k}$. 
States with a larger component in this direction are not relevant in this case. In order to reach a state with a large component in the eigenspaces of $P$ belonging to the small eigenvalues, a larger control needs to be used. 
A similar estimate as in Corollary \ref{correach} has already been obtained for a different reachability Gramian \cite{redstochbil,redmanntypeiibilinear}. \smallskip

Based on the result in Theorem \ref{mainthm}, a bound for the output error between systems (\ref{sys:original}) and (\ref{sys:reduced}) is derived now.
\begin{kor}\label{corh2error}
Let $y$ be the output of system (\ref{sys:original}) with $x_0=0$ satisfying (\ref{assumption}). Moreover, let $\hat y$ be the output of the reduced system (\ref{sys:reduced}) with $\hat x_0=0$ and
 \begin{align}\label{assumptionred}
 \sigma(\hat A\otimes I+I\otimes \hat A+ \sum_{k=1}^m \hat N_k \otimes \hat N_k)\subset \mathbb C_-.  
\end{align}
Then, we have \begin{align}\label{rombound}
  \sup_{t\geq 0}\left\|y(t) - \hat y(t)\right\|_2 \leq \left(\trace(C P C^T) + \trace(\hat C \hat P \hat C^T)- 2 \trace(C P_g \hat C^T)\right)^{\frac{1}{2}} \exp\left\{0.5\left\|u^{0}\right\|_{L^2}^2\right\} \left\|u\right\|_{L^2},                   
                    \end{align}
where $P$ solves (\ref{lyapeq}) and $\hat P$, $P_g$ are the solutions to  \begin{align}\label{lyapeqred}
                  \hat A \hat P + \hat P \hat A^T + \sum_{k=1}^m \hat N_k \hat P \hat N_k^T &= - \hat B \hat B^T,\\ \label{lyapmix}
                  A P_g + P_g \hat A^T + \sum_{k=1}^m N_k P_g \hat N_k^T &= - B \hat B^T.
                  \end{align}
The trace expression in (\ref{rombound}) represents the $\mathcal{H}_2$-error between systems (\ref{sys:original}) and (\ref{sys:reduced}), i.e., \begin{align*}
\left\|\Sigma - \hat \Sigma\right\|_{\mathcal H_2}^2 = \trace(C P C^T) + \trace(\hat C \hat P \hat C^T)- 2 \trace(C P_g \hat C^T).
\end{align*}
\end{kor}
\begin{proof}
We define the error state $x^e$ and the error matrices $(A^e, B^e, C^e, N_k^e)$ as follows:
\begin{align*}
x^e =\smat x \\ \hat x \srix,\; A^e = \smat{A}& 0\\ 
0 &\hat A\srix,\;B^e=\smat B \\ \hat B\srix,\;C^e= \smat C & -\hat C \srix,\; N_k^e=\smat {N}_k& 0 \\ 
0 &\hat N_k\srix.
            \end{align*} 
It is not hard to see that $x^e$ satisfies the state equation of system \begin{equation*}
\begin{aligned}
\dot x^e (t) &= A^e x^e(t) +  B^e u(t) + \sum_{k=1}^m N_k^e x^e(t) u_k(t),\quad x^e(0) = 0,\\
y^e(t) &= C^e x^e(t),\quad t\geq 0,
\end{aligned}
\end{equation*}
and the corresponding output $y^e$ coincides with the output error between (\ref{sys:original}) and (\ref{sys:reduced}), i.e., $y^e = y - \hat y$. With the same steps as in the proof of Theorem \ref{mainthm}, we find 
analogous to (\ref{lastest}) that \begin{align}\label{blaaaa}
\left\|y^e(t)\right\|_2\leq \exp\left\{0.5 \int_0^t \left\|u^{0}(s)\right\|_2^2 ds\right\} \left(\trace(C P^e_t C^T)\right)^{\frac{1}{2}} \left(\int_0^t \left\|u(s)\right\|_2^2 ds\right)^{\frac{1}{2}},
\end{align}
where $P^e_t = \int_0^t \bar Z^e(s) ds$ with $\bar Z^e$ satisfying
\begin{align}\label{matrix_eq_sinu_error}
 \dot {\bar Z}^e(t) = A^e {\bar Z}^e(t) + {\bar Z}^e(t) (A^e)^T +\sum_{k=1}^m N_k^e {\bar Z}^e(t) (N^e_k)^T ,\quad {\bar Z}^e(0) = B^e(B^e)^T.
\end{align}
To prove the existence of $P^e = \int_0^\infty {\bar Z}^e(s) ds$, we partition ${\bar Z}^e = \smat {\bar Z}_{11}& {\bar Z}_{12}\\ {\bar Z}_{12}^T & {\bar Z}_{22}\srix$. Evaluating the left upper and the right lower block of 
(\ref{matrix_eq_sinu_error}), we see that ${\bar Z}_{11}$ solves (\ref{matrix_eq_sinu}) and ${\bar Z}_{22}$ the same equation, where $(A, B, N_k)$ are replaced by the reduced coefficients $(\hat A, \hat B, \hat N_k)$. With 
the arguments of the proof of Theorem \ref{mainthm}, the assumptions (\ref{assumption})  and (\ref{assumptionred}) imply $\left\|{\bar Z}_{11}(t)\right\|_2, \left\|{\bar Z}_{22}(t)\right\|_2\lesssim \expn^{-k t}$ for 
some $k>0$. Since ${\bar Z}^e$ is positive semidefinite, it therefore holds that \begin{align*}
  \left\|{\bar Z}^e(t)\right\|_2 &=   \left\|[{\bar Z}^e(t)]^{\frac{1}{2}} [{\bar Z}^e(t)]^{\frac{1}{2}}\right\|_2   \leq \left\|[{\bar Z}^e(t)]^{\frac{1}{2}}\right\|_2^2  \leq \left\|[{\bar Z}^e(t)]^{\frac{1}{2}}\right\|^2_F  
  =\trace\left({\bar Z}^e(t)\right) \\&= \trace\left({\bar Z}_{11}(t)\right) + \trace\left({\bar Z}_{22}(t)\right) \lesssim \expn^{-k t}.
                                                     \end{align*}
This gives us the existence of $P^e = \smat P_{11}& P_{12}\\ 
P_{12}^T & P_{22}\srix$ which according to Theorem \ref{mainthm} satisfies 
\begin{align}\label{errorlyapeq}
                  A^e P^e + P^e (A^e)^T + \sum_{k=1}^m N^e_k P^e (N^e_k)^T = - B^e (B^e)^T.
                                                                                              \end{align}
Evaluating the respective blocks of (\ref{errorlyapeq}), we find $P_{11} = P$, $P_{12} = P_g$ and $P_{22} = \hat P$. We can now take the supremum on both sides of (\ref{blaaaa}) and obtain
\begin{align*}
  \sup_{t\geq 0}\left\|y^e(t)\right\|_2 \leq \left(\trace(C^e P^e (C^e)^T)\right)^{\frac{1}{2}} \exp\left\{0.5\left\|u^{0}\right\|_{L^2}^2\right\} \left\|u\right\|_{L^2}.                  
                    \end{align*}
Using the partitions of $C^e$ and $P^e$, we have $\trace(C^e P^e (C^e)^T) = \trace(C P C^T) + \trace(\hat C \hat P \hat C^T)- 2 \trace(C P_g \hat C^T)$. It was discussed in \cite{morZhaL02} that 
this expression coincides with $\left\|\Sigma - \hat \Sigma\right\|_{\mathcal H_2}^2$. This concludes the proof.
\end{proof}\\
Notice that the bound in (\ref{rombound}) can potentially be large due to the exponential term if the control energy is large. This can, e.g., 
happen if the original system has to be rescaled with a constant $\gamma$ (see Remark \ref{remark1}) in order to guarantee (\ref{assumption}).
We can also see from Corollary \ref{corh2error} that control components $u_k$ with $N_k \ne 0$ have a much larger impact on the bound because their energy enters exponentially. 
Later we will discuss balancing related MOR schemes and prove their error bounds based on (\ref{rombound}). For those methods (\ref{assumptionred}) usually 
follows automatically from (\ref{assumption}).

\subsection{Asymptotic stability in bilinear systems}

We conclude this section with a probabilistic proof on asymptotic stability for bilinear systems, which is a consequence of Lemmas \ref{lem1}, \ref{lem2} and \ref{lem3}.
\begin{thm}\label{asymptstabstoch}
Let $x(t, x_0, 0)$, $t\geq 0$, denote the solution to the homogeneous bilinear equation \begin{align}\label{stabhosys}
\dot x(t) = Ax(t) + \sum_{k=1}^m N_k x(t) u_k(t),\quad x(0)=x_0.
\end{align}
If $\sigma(A)\subset \mathbb C_-$, then there exit $\gamma, k_1, k_2 >0$ such that \begin{align*}
\left\|x(t, x_0, 0)\right\|_2^2  \leq  \exp\left\{\gamma^2\left\|u^{0}\right\|_{L^2}^2\right\} \left\|x_0\right\|_2^2 k_1 \expn^{-k_2 t}                                                                                   
                                                                                    \end{align*}
for all $u\in L^2$, i.e., the bilinear equation is asymptotically stable with exponential decay.                                                                                  
\end{thm}
\begin{proof}
We can equivalently rewrite equation (\ref{stabhosys}) as \begin{align*}
   \dot x(t) = Ax(t) + \sum_{k=1}^m [\frac{1}{\gamma} N_k] x(t) [\gamma u_k(t)]                                                        
                                                          \end{align*}
as explained in Remark \ref{remark1}. We set $\tilde N_k:=\frac{1}{\gamma} N_k$. We define the corresponding stochastic equation by \begin{align*}
   d\tilde z(t) = A\tilde z(t)dt + \sum_{k=1}^m \tilde N_k \tilde z(t) dw_k(t),\quad \tilde z(0)=x_0,                                                       
                                                          \end{align*}
and denote its solution by $\tilde z(t, x_0)$. Lemmas \ref{lem1}, \ref{lem2}, \ref{lem3} and Proposition \ref{prop_stoch_rep} imply the relation between $x$ and $\tilde z$ given in (\ref{essentialest_stoch}),
only $u$ is replaced by the rescaled input $\gamma u$. Applying the trace operator to both sides of (\ref{essentialest_stoch}), the inequality is preserved and we obtain
\begin{align*}
x^T(t, x_0, 0) x(t, x_0, 0)\leq \exp\left\{ \gamma^2 \int_0^t\left\|u^{0}(\tau)\right\|_2^2 d\tau\right\}\mathbb E\left[\tilde z^T(t, x_0) \tilde z(t, x_0)\right].
\end{align*}
We enlarge the right side of this inequality by replacing $\int_0^t\left\|u^{0}(\tau)\right\|_2^2d\tau$ by $\left\|u^{0}\right\|_{L_2}^2$. 
Now we can choose $\gamma > \sqrt{\sum_{k=1}^m \left\| N_k \right\|_2^2 \int_0^\infty \left\|\expn^{A t}\right\|_2^2  dt}$. According to Remark \ref{remark1} this implies (\ref{scalestab}). By Lemma \ref{lemstab},
we therefore know that there exist $k_1, k_2>0$, such that 
\begin{align*}
\mathbb E \left[\tilde z^T(t, x_0) \tilde z(t, x_0)\right] \leq \left\|x_0\right\|^2_2 k_1 \expn^{-k_2 t}.
\end{align*}
This concludes the proof.
\end{proof}\\
It is interesting to notice that adding a bilinearity to an asymptotically stable linear system $\dot x(t) = Ax(t)$ preserves this stability condition. The additional bilinear term only enlarges the constant but does not change the decay. Notice that the result of Theorem \ref{asymptstabstoch} can be proved using the deterministic concept of integral-input to state stability (IISS) \cite{sontag}. IISS is equivalent to $\sigma(A)\subset \mathbb C_-$ in bilinear systems and implies Theorem \ref{asymptstabstoch} if $u\in L^2$. The result can be extended to $x(t, x_0, 0)\rightarrow 0$ as $t\rightarrow \infty$ for any $u\in L^q$, $1\leq q <\infty$, using IISS arguments. 

\section{Consequences of the $\mathcal H_2$-bound in the context of MOR}

The bound in Corollary \ref{corh2error} has various important consequences for MOR schemes applied to bilinear systems. Based on this bound, we are able to explain why $\mathcal{H}_2$-optimal MOR techniques lead to small output 
errors. Moreover, we can prove output error bounds for both BT and SPA which allow us to point out the cases in which both methods yield a good approximation.
These error bounds are derived from known results in the error analysis for stochastic systems by the link that Corollary \ref{corh2error} provides. 
Throughout this section, we will assume (\ref{assumption}). We know that this is achieved if $\sigma(A)\subset \mathbb C_-$ and the bilinear system is rescaled according to Remark \ref{remark1}. Moreover, we assume 
(\ref{assumptionred}) if it is not automatically given through (\ref{assumption}). This then guarantees existence of the bound in Corollary \ref{corh2error}.

\subsection{$\mathcal H_2$-optimal MOR}\label{sec:optimality_condtions}

The error bound in Corollary \ref{corh2error} shows that we can expect a small output error if we find a reduced system that leads to a small 
$\mathcal E^2 := \trace(C P C^T) + \trace(\hat C \hat P \hat C^T)- 2 \trace(C P_g \hat C^T)$ (in case the control $u$ is not too large). Consequently, one can expect a good ROM when 
$\mathcal E$ is minimized with respect to $\hat A, \hat B, \hat C$ and $\hat N_k$. As mentioned in the previous section, 
$\mathcal E$ is the $\mathcal H_2$-error between systems (\ref{sys:original}) and (\ref{sys:reduced}). Necessary conditions for a locally optimal $\mathcal H_2$-error have already been provided \cite{morZhaL02}. 
These are \begin{align}\label{optcond}
           \hat C \hat P = C P_{g},\quad \hat Q\hat B  =  Q_{g} B,\quad \hat Q\hat  P = Q_{g} P_{g},\quad \hat Q \hat N_k \hat  P = Q_{g} N_k P_{g},    
          \end{align}
$k\in\left\{1,\ldots, m\right\}$, where $\hat P, P_g$ are the solutions to (\ref{lyapeqred}) and (\ref{lyapmix}). Moreover, $\hat Q, Q_g$ satisfy 
 \begin{align*}
                  \hat A^T \hat Q + \hat Q \hat A + \sum_{k=1}^m \hat N_k^T \hat Q \hat N_k &= - \hat C^T \hat C,\\ 
                  \hat A^T Q_g + Q_g A + \sum_{k=1}^m \hat N_k^T Q_g N_k &= - \hat C^T C.
                  \end{align*}
By Corollary \ref{corh2error} it is now clear that choosing reduced systems (\ref{sys:reduced}) satisfying (\ref{optcond}) is meaningful in terms of the output error. This is a new insight since the link between the 
output and the $\mathcal H_2$-error was not previously known. Such $\mathcal H_2$-optimal ROMs are, e.g., derived through generalized Sylvester iterations, see Algorithm 1 in \cite{morBenB12b}.
Another very famous representative of $\mathcal H_2$-optimal schemes is the bilinear iterative rational Krylov algorithm (IRKA), see Algorithm \ref{algo:IRKA}. Due to a reformulation of (\ref{optcond}), it could be shown in 
\cite{morBenB12b} that bilinear IRKA satisfies the necessary optimality conditions.
\begin{algorithm}[!tb]
	\caption{Bilinear IRKA}
	\label{algo:IRKA}
	\begin{algorithmic}[1]
		\Statex {\bf Input:} The system matrices: $ A, B, C, N_k$.
		\Statex {\bf Output:} The reduced matrices: $\hat A, \hat B,\hat C, \hat N_k$.
		\State Make an initial guess for the reduced matrices $\hat A, \hat B,\hat C, \hat N_k$.
\While {not converged}
		\State Perform the spectral decomposition of $\hA$ and define:
		\Statex\quad\qquad $D = S\hA S^{-1},~\tB = S\hB, ~\tC = \hC S^{-1}, ~\tN_k = S\hN_k S^{-1}. $
		\State Solve for $V$ and $W$:
		
		\Statex \quad\qquad$ -V D  -  AV -\sum_{k=1}^m N_k V \tN_k^T  = B\tB^T$,
		\Statex \quad\qquad$ -W D  -  A^T W -\sum_{k=1}^m N_k^T W \tN_k = C^T\tC$.
\State $V = \orth{(V)}$ and $W = \orth{(W)}$, where $\orth{(\cdot)}$ returns an orthonormal basis for the range of a matrix.
		\State Determine the reduced matrices:
		\Statex \quad\qquad$\hA = (W^T V)^{-1}W^T AV,\;  \hB = (W^T V)^{-1}W^T B,\;\hC = CV$, \; $\hN_k = (W^T V)^{-1}W^T N_kV$.
		\EndWhile
	\end{algorithmic}
\end{algorithm}

\subsection{Error bounds for balancing related MOR techniques applied to bilinear systems}\label{sectionBT_SPA}

We explain the procedure of balancing related MOR first, before we show the error bounds for two particular methods. These are balanced truncation (BT) and singular perturbation approximation (SPA).
States that require a lager amount of energy to be reached (hard to reach states) can be identified through the Gramian 
$P$ solving (\ref{lyapeq}) using Corollary \ref{correach}. We refer to \cite{morBenD11, enefungray98} for alternative characterizations based on $P$ and to \cite{redstochbil,redmanntypeiibilinear}
for estimates based on an alternative reachability Gramian. States that produce only a small amount of observation energy (hard to observe states)
can be found through an observability Gramian $Q$ \cite{morBenD11, enefungray98, redstochbil}  satisfying 
\begin{align}\label{lyapeqdual}
                  A^T Q + Q A + \sum_{k=1}^m N_k^T Q N_k = - C^T C.
                                                                                              \end{align}
The goal is to remove the hard to reach and observe states that are contained in the eigenspaces of $P$ and $Q$, respectively, corresponding to the small eigenvalues. This is done by 
simultaneously diagonalizing $P$ and $Q$ such that they are equal and diagonal. Subsequently, the states contributing only very little to the systems dynamics are neglected.\smallskip

In detail, the procedure works as as follows. Assuming $P, Q >0$, we choose a regular state space transformation $S\in \mathbb R^n$ given by 
\begin{align}\label{balancedtrans}
S=\Sigma^{-\frac{1}{2}}U^T L^T \quad\text{and}\quad S^{-1}=KV\Sigma^{-\frac{1}{2}},
\end{align}
where $\Sigma=\diag(\sigma_{1},\ldots,\sigma_n)>0$ with diagonal entries being the square root of eigenvalues of $PQ$. These diagonal entries are called Hankel singular values (HSVs). The other ingredients of the transformation 
$S$ are computed in the following way. Let us factorize $P = KK^T$ and $Q=LL^T$, then a singular value decomposition of $K^TL = V\Sigma U^T$ gives the required matrices. We now introduce a transformed state 
\begin{align*}x_b(t) = Sx(t) = \smat x_1(t) \\x_2(t)\srix,\end{align*}
where $x_1(t)\in\mathbb R^r$. The transformed state $x_b$ satisfies a bilinear system with the same output as (\ref{sys:original}) having the following coefficients\begin{align}\label{partmatrices}
S{A}S^{-1}= \smat{A}_{11}&{A}_{12}\\ 
{A}_{21}&{A}_{22}\srix,\quad S{B} = \smat{B}_1\\ {B}_2\srix,\quad  
{CS^{-1}} = \smat{C}_1 &{C}_2\srix,\quad S{N_k}S^{-1}= \smat {N}_{k, 11}&{N}_{k, 12}\\ 
{N}_{k, 21}&{N}_{k, 22}\srix, \end{align}
where ${A}_{11}\in\R^{r\times r}$ etc. Using the above partitions, this system is \begin{equation}\label{sys:original_trans}
\begin{aligned}
\smat {\dot x}_1(t) \\ {\dot x}_2(t)\srix &= \smat{A}_{11}&{A}_{12}\\ 
{A}_{21}&{A}_{22}\srix \smat x_1(t) \\x_2(t)\srix +  \smat{B}_1\\ {B}_2\srix u(t) + \sum_{k=1}^m \smat {N}_{k, 11}&{N}_{k, 12}\\ 
{N}_{k, 21}&{N}_{k, 22}\srix \smat x_1(t) \\x_2(t)\srix u_k(t),\\
y(t) &= \smat{C}_1 &{C}_2\srix \smat x_1(t) \\x_2(t)\srix,\quad t\geq 0.
\end{aligned}
\end{equation}
The Gramians $P_b$ and $Q_b$ of (\ref{sys:original_trans}) are  \begin{align*}
P_b =S P S^T = Q_b = S^{-T}QS^{-1}= \Sigma = \smat{\Sigma}_{1}& \\ 
 &{\Sigma}_{2}\srix
\end{align*}
with ${\Sigma}_{1}\in\R^{r\times r}$. ${\Sigma}_{2} = \diag(\sigma_{r+1},\ldots,\sigma_n)$ contains the $n-r$ smallest HSVs of the systems. The corresponding state variables $x_2$ are hence less important 
and can be neglected in the system dynamics since those represent the difficult to reach and observe states in (\ref{sys:original_trans}). In order to obtain a ROM, the second line in the state equation 
of (\ref{sys:original_trans}) is truncated. The remaining $x_2$ components in the first line of the state equation and in the output equation can now be approximated in two ways. One is setting 
$x_2(t) = 0$. This method is called BT and the ROM (\ref{sys:reduced}) then has coefficients
\begin{align}\label{btmat}
(\hat A, \hat  B, \hat C,\hat  N_{k})=(A_{11}, B_1, C_1, N_{k, 11}).
\end{align}
An alternative method is SPA where one sets $x_2(t) = - A_{22}^{-1} A_{21} x_1(t)$. This results in a ROM with matrices
\begin{align}\label{spamat}
(\hat A, \hat  B, \hat C,\hat  N_{k})=(\bar A, B_1, \bar C, \bar N_k),
\end{align}
where we define
\begin{align*}
 \bar A:=A_{11}- A_{12} A_{22}^{-1} A_{21},\quad \bar C:=C_1-C_2 A_{22}^{-1} A_{21},\quad \bar N_k:=N_{k,11}-N_{k, 12} A_{22}^{-1} A_{21}.
      \end{align*}
We refer to \cite{hartmann} for more details about SPA. There, the reduced system with matrices as in (\ref{spamat}) is derived through an averaging principle. In the following,
we present $L^\infty$-error bounds for both BT and SPA. Both results are new and the first ones for balancing related methods based on the $\mathcal H_2$-error.
However, we want to mention that there is an alternative
bound for BT in infinite dimensions \cite{beckerhartmann} and there are $L^2$-error bounds for BT and SPA based on a different reachability
Gramian $P_2$ \cite{redstochbil, redmannstochbilspa}. $P_2$ is defined to be a positive definite solution to \begin{align}\label{reachlyaptype2}
 A^T P_2^{-1}+P_2^{-1}A+\sum_{k=1}^m N_k^T P_2^{-1} N_k \leq -P_2^{-1}BB^T P_2^{-1}. \end{align}
Replacing $P$ by $P_2$ is also called type II approach. \smallskip 

For simplicity of the notation, we from now on assume that system (\ref{sys:original}) is already balanced, i.e., 
we already applied the balancing transformation in (\ref{balancedtrans}) such that $P = Q = \Sigma = \diag(\sigma_{1},\ldots,\sigma_n)$. We formulate the error bound for BT first.
\begin{thm}[Error bound BT]\label{theoremerrorboundBT}
Let $(A, B, C, N_k)$ be a balanced realization of system (\ref{sys:original}) with partitions as in (\ref{partmatrices}). Suppose that 
$\hat y$ is the output of the reduced system with matrices given in (\ref{btmat}). Moreover, we assume that (\ref{assumption}) holds. Then, given zero initial conditions for both the full and the reduced system, we have
\begin{align*}
  \sup_{t\geq 0}\left\|y(t)-\hat y(t)\right\|_2 \leq \left(\trace(\Sigma_2 K_{BT})\right)^{\frac{1}{2}} \exp\left\{0.5\left\|u^{0}\right\|_{L^2}^2\right\} \left\|u\right\|_{L^2}                   
                    \end{align*}
with $\Sigma_2=\diag(\sigma_{r+1}, \ldots, \sigma_n)$ and the weighting matrix given by\begin{align*}
K_{BT} =  B_2 B_2^T+2  P_{g, 2}A_{21}^T +\sum_{k=1}^m (2 N_{k, 22} P_{g, 2}N_{k, 21}^T+2 N_{k, 21} P_{g, 1}N_{k, 21}^T-N_{k, 21} \hat P  N_{k, 21}^T),
\end{align*}
where $\hat P$ and $P_{g}=\smat P_{g, 1}\\ P_{g, 2}\srix$ satisfy (\ref{lyapeqred}) and (\ref{lyapmix}), respectively.
\end{thm}
\begin{proof}
Assumption (\ref{assumption}) implies (\ref{assumptionred}) in the case of BT, see \cite{redbendamm}, i.e., mean square asymptotic stability is preserved in the ROM. Consequently, the bound in 
Corollary \ref{corh2error} exists. So, it holds that \begin{align*}
  \sup_{t\geq 0}\left\|y(t) - \hat y(t)\right\|_2 \leq \left(\trace(C P C^T) + \trace(\hat C \hat P \hat C^T)- 2 \trace(C P_g \hat C^T)\right)^{\frac{1}{2}} \exp\left\{0.5\left\|u^{0}\right\|_{L^2}^2\right\} \left\|u\right\|_{L^2}.                   
                    \end{align*}
The trace expression in the above estimate has already been analyzed within the error bound analysis of stochastic systems. By \cite[Proposition 4.6]{redmannbenner}, we then have \begin{align*}
 \trace(C P C^T) + \trace(\hat C \hat P \hat C^T)- 2 \trace(C P_g \hat C^T) = \trace(\Sigma_2 K_{BT}).                                                                                                                                                                                   
                                                                                                                                                                                   \end{align*}
                   \end{proof}\\
Setting $N_k=0$ in Theorem \ref{theoremerrorboundBT} leads to the $\mathcal H_2$-error bound in the linear case \cite{morAnt05}. We now state the bound for SPA.                   
\begin{thm}[Error bound SPA]\label{theoremerrorboundSPA}
Let $(A, B, C, N_k)$ be a balanced realization of system (\ref{sys:original}) with partitions as in (\ref{partmatrices}). Suppose that 
$\hat y$ is the output of the reduced system with matrices given in (\ref{spamat}). Moreover, we assume that (\ref{assumption}) and  \begin{align}\label{stabpreservcond}
 0 \not\in \sigma(\hat A\otimes I+I\otimes \hat A+ \sum_{k=1}^m \hat N_k \otimes \hat N_k)  
\end{align}
hold. Then, zero initial conditions for both the full and the reduced system yield
\begin{align*}
  \sup_{t\geq 0}\left\|y(t)-\hat y(t)\right\|_2 \leq \left(\trace(\Sigma_2 K_{SPA})\right)^{\frac{1}{2}} \exp\left\{0.5\left\|u^{0}\right\|_{L^2}^2\right\} \left\|u\right\|_{L^2}                   
                    \end{align*}
with $\Sigma_2=\diag(\sigma_{r+1}, \ldots, \sigma_n)$ and the weighting matrix given by\begin{align*}
K_{SPA} =  & B_2 B_2^T-2(A_{22}P_{g, 2}+A_{21} P_{g, 1})(A_{22}^{-1} A_{21})^T\\
&+ 2\sum_{k=1}^m (N_{k, 22} P_{g, 2}+N_{k, 21}P_{g, 1})(N_{k, 21}-N_{k, 22}A_{22}^{-1} A_{21})^T \\
&- \sum_{k=1}^m (N_{k, 21}-N_{k, 22}A_{22}^{-1} A_{21}) \hat P  (N_{k, 21}-N_{k, 22}A_{22}^{-1} A_{21})^T,
\end{align*}
where $\hat P$ and $P_{g}=\smat P_{g, 1}\\ P_{g, 2}\srix$ satisfy (\ref{lyapeqred}) and (\ref{lyapmix}), respectively.
\end{thm}
\begin{proof}
Assumptions (\ref{assumption}) and (\ref{stabpreservcond}) yield (\ref{assumptionred}) for SPA, see \cite{redSPA}. This guarantees existence of the bound in Corollary \ref{corh2error}. Consequently, we have \begin{align*}
  \sup_{t\geq 0}\left\|y(t) - \hat y(t)\right\|_2 \leq \left(\trace(C P C^T) + \trace(\hat C \hat P \hat C^T)- 2 \trace(C P_g \hat C^T)\right)^{\frac{1}{2}} \exp\left\{0.5\left\|u^{0}\right\|_{L^2}^2\right\} \left\|u\right\|_{L^2}.                   
                    \end{align*}
The above term is known for SPA applied to stochastic systems. By \cite[Theorem 4.1]{redSPA}, we have \begin{align*}
 \trace(C P C^T) + \trace(\hat C \hat P \hat C^T)- 2 \trace(C P_g \hat C^T) = \trace(\Sigma_2 K_{SPA}).                                                                                                                                                                                   
                                                                                                                                                                                   \end{align*}
                   \end{proof}\\
Theorems \ref{theoremerrorboundBT} and \ref{theoremerrorboundSPA} show us in which cases BT and SPA work well. If one only truncates the states corresponding to the small Hankel singular values (hard to reach and observe states),
then the diagonal entries $\sigma_{r+1}, \ldots, \sigma_n$ of $\Sigma_2$ and hence the output error are small assuming that the control is not too large.  
\begin{remark} 
Based on the results in \cite{BTtyp2EB} and \cite{redmannspa2} Theorems \ref{theoremerrorboundBT} and \ref{theoremerrorboundSPA} can be formulated the same way if $P$ is replaced by $P_2$ satisfying (\ref{reachlyaptype2}). 
In this type II case,(\ref{assumptionred}) automatically follows from (\ref{assumption}) for BT and SPA due to \cite{bennerdammcruz} and \cite{redmannspa2}. The reason why (\ref{stabpreservcond}) has to be assumed above is that 
stability preservation for SPA based on $P$ is still an open question.
\end{remark}

\section{Numerical experiments}

In this section, we conduct numerical experiments in order to estimate the quality of the bounds in Theorem \ref{mainthm} and the associated Corollary \ref{corh2error}. We investigate the influence of the control 
energy on the sharpness of the result in Theorem \ref{mainthm}. Subsequently, we apply the bound in Corollary \ref{corh2error} in the context of MOR and show which impact a rescaling factor $\gamma$ has that we 
discussed in Remark \ref{remark1}.
\subsection{The output bound depending on the control energy}\label{sec61}

In the proof of Theorem \ref{mainthm}, the Gronwall Lemma \ref{lem3} is involved which is not generally tight. 
In particular one step in this proof is to bound $x_{(s)}(t, b_j, 0)x^T_{(s)}(t, b_j, 0)$ which satisfies (\ref{matrix_ineq}) by the 
solution of (\ref{matrix_eq}) from above. Here, $b_j$ denotes the $j$th column of the input matrix $B$.
It can be expected that the smaller the gap between the left and the right side of inequality (\ref{matrix_ineq}), the better the Gronwall estimate. In the last step in Appendix \ref{secprooflem1}, we can see that the gap is exactly 
\begin{align}\label{gap_to_eq}
\sum_{k=1}^m \left(N_k - u^0_k(t)I\right) x_{(s)}(t, b_j, 0)x^T_{(s)}(t, b_j, 0) \left(N_k - u^0_k(t) I\right)^T. 
\end{align}
By assumption (\ref{assumption}), $N_k$ is usually small meaning that (\ref{gap_to_eq}) is expected to be small if the control vector $u^0$ is not too large which furthermore implies smallness of $x_{(s)}(t, b_j, 0)$. We can also 
see the impact of $B$ because it influences $x_{(s)}(t, b_j, 0)$ and hence also (\ref{gap_to_eq}). \smallskip

This intuition is confirmed by the following numerical example \begin{align*}
\dot x(t) &= Ax(t) +  b_1 u_1(t)+ b_2 u_2(t) + N_1 x(t) u_1(t),\quad x(0) = 0,\\
y(t) &= Cx(t),                                                              
                                                               \end{align*}
where we choose \begin{align*}
A=    \mat{cc}-2  & 1\\   1 & -2\rix, \quad N_1= \mat{cc}0  & 1\\   0.5 & 0\rix, \quad B= \mat{cc}1  & 0\\   0 & 1\rix, \quad C=    \mat{cc}  1 &   1 \rix.
                \end{align*}
This example satisfies (\ref{assumption}) and we control the above system on $[0, 1]$ meaning that
\begin{align}\label{control_numerics}
 u(t) = (u_1(t)\ u_2(t))^T = \alpha\, \bar u(t)/\left\|\bar u\right\|_{L^2}, \quad \bar u(t) = \begin{cases} (\expn^{-t} \ \sin(\pi t) \expn^t)^T & \text{if } t\in [0, 1]\\ (0\ 0)^T & t>1. \end{cases}   
\end{align}
We compute $\sup_{t\in[0, T]} y(t)$, $T\geq 1$, and the corresponding bound from Theorem \ref{mainthm} for several $\alpha= \left\|u\right\|_{L^2}\in[0, 4]$, where we set
$f(u): = \exp\left\{0.5\left\|u^{0}\right\|_{L^2}^2\right\} \left\|u\right\|_{L^2}$. Moreover, notice that $\left\|u^{0}\right\|_{L^2} = \left\|u_1\right\|_{L^2} \approx 0.48 \alpha$. 
In Figure \ref{alpha12}, we see that the bound is a very good approximation if $\left\|u\right\|_{L^2}\leq 2$ and it performs best around a control energy of around $0.5$. In Figure \ref{alpha34}, we observe 
that with increasing energy, the bound gets less sharp but it is still acceptable if $\left\|u\right\|_{L^2}\leq 4$. However, the tendency of the graph in Figure \ref{alpha34} already indicates that 
the bound from Theorem \ref{mainthm} is not accurate for large controls.\smallskip

\begin{figure}[ht]
\begin{minipage}{0.45\linewidth}
 \hspace{-0.5cm}
\includegraphics[width=1.1\textwidth]{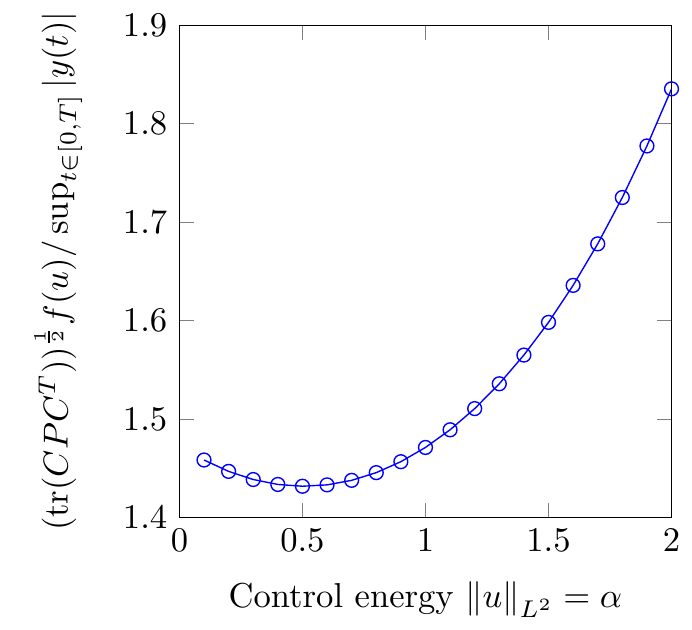}
\caption{Bound $(\trace(C P C^T))^{\frac{1}{2}} f(u)$ divided by $\sup_{t\in [0, T]} \vert y(t)\vert$ for $u$ defined in (\ref{control_numerics}), control energies $\alpha\in [0, 2]$ and $T=2$.}\label{alpha12}
\end{minipage}\hspace{0.5cm}
\begin{minipage}{0.45\linewidth}
 \hspace{-0.5cm}
\includegraphics[width=1.1\textwidth]{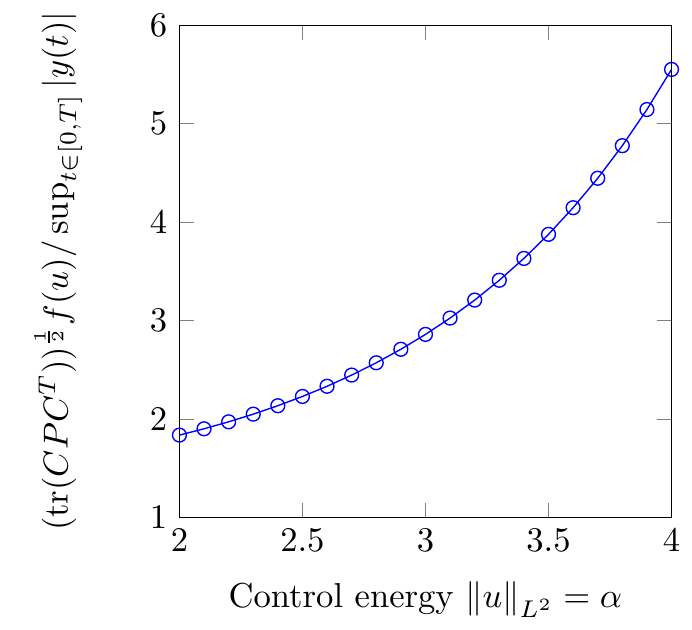}
\caption{Bound $(\trace(C P C^T))^{\frac{1}{2}} f(u)$ divided by $\sup_{t\in [0, T]} \vert y(t)\vert$ for $u$ defined in (\ref{control_numerics}), control energies $\alpha\in [2, 4]$  and $T=2$.}\label{alpha34}
\end{minipage}
\end{figure}
An immediate improvement of the bound (for large $u$) can be seen by considering the time-limited Gramian $P_t$, defined above (\ref{lastest}), instead of $P=\lim_{t\rightarrow \infty} P_t$, since $P_t\leq P$. The associated estimate is stated in 
(\ref{lastest}). Moreover, $P_t$ neither needs (\ref{assumption}) nor $\sigma(A)\subset \mathbb C_-$ for its existence which also avoids a rescaling with a constant $\gamma$, see Remark \ref{remark1}. However, how to compute 
$P_t$ for large $n$, is an open question and rather challenging.\smallskip

Another solution could be truncated Gramians which lead to truncated $\mathcal{H}_2$-norms \cite{bengoygug, flagg2015multipoint}. Their benefit is that they can be computed more easily than $P$ and they only require 
$\sigma(A)\subset \mathbb C_-$ but not (\ref{assumption}) for their existence which avoids scaling the systems with $\gamma$. However, it is not yet clear how to derive an output bound based on truncated Gramians. If it exists, it requires 
different techniques than used here. We address the influence of the scaling factor $\gamma$ in the next section.

\subsection{The output bound in the context of MOR}

We apply BT explained in Section \ref{sectionBT_SPA} to a modified version of a heat transfer problem considered in \cite{morBenD11}. In particular, we reduce the dimension 
of a spatially discretized heat equation and compute the error bound of Theorem \ref{theoremerrorboundBT} by the representation given in Corollary \ref{corh2error}. Notice that within this example, a rescaling of the 
resulting bilinear system by a constant $\gamma$ is needed, see Remark \ref{remark1}, in order to meet the assumptions for the existence of the error bound. We investigate the performance of the bound depending on the 
rescaling factor below.\smallskip

Let us study the following heat equation on $[0, 1]^2$:
\begin{align*}
\frac{\partial}{\partial t} X(t,\zeta)=\Delta X(t, \zeta),\quad \zeta\in [0, 1]^2,\quad t\in [0, T],
\end{align*}
with mixed Dirichlet and Robin boundary conditions \begin{align*}
  \frac{\partial}{\partial \mathbf{n}} X(t,\zeta)&= u_1(t) (X(t, \zeta) -1)\quad \text{on}\quad \Gamma_1:=\left\{0\right\}\times (0, 1),    \\    
  X(t,\zeta)&= u_2(t)\quad \text{on}\quad \Gamma_2:=\left\{1\right\}\times (0, 1),    \\
   X(t,\zeta)&= 0 \quad \text{on}\quad \partial [0, 1]^2\setminus (\Gamma_1 \cup \Gamma_2),\quad t\in [0, T],
                                                   \end{align*}
where $u = (u_1(t)\ u_2(t))^T$ denotes the input to the system. Like in \cite{morBenD11} we discretize the heat equation with a finite difference scheme on an equidistant
$\tilde n \times \tilde n$-mesh. This leads to an $n=\tilde n^2$-dimensional bilinear system 
\begin{equation}\label{bilisysnum}
 \begin{split}
             \dot x(t) &=Ax(t)+Bu(t)+N_1 x(t)u_1(t),\\ 
             y(t)&=Cx(t), \quad t\in [0, T],
            \end{split}
                       \end{equation}
where $N_2 = 0$ and $C= \frac{1}{n} [\begin{matrix}1& 1& \dots & 1\end{matrix}]$, i.e., the quantity of interest is the average temperature. We refer to \cite{morBenD11} for more details on the matrices $A, B$ and $N_1$. \smallskip

We choose $n=900$ ($\tilde n = 30$) and observe that (\ref{bilisysnum}) satisfies $\sigma(A) \subset \C_{-}$ but not (\ref{assumption}). 
We can equivalently rewrite (\ref{bilisysnum}) as in (\ref{rescaled_sys}) and choose $\gamma\geq 1.3$ to achieve (\ref{scalestab}). 
Now, BT is applied based on the rescaled control $u\mapsto \gamma u$ and the matrices $(B, N_1)\mapsto \frac{1}{\gamma}(B, N_1)$ in order to derive the reduced system (\ref{sys:reduced}) with state space
dimension $r=10$. Subsequently, we compute the 
error bound in Corollary \ref{corh2error}, where the trace expression there is denoted by $\mathcal E_\gamma$ such that the result in this corollary is written as \begin{align*}
       \sup_{t\in [0, T]}\vert y(t) - \hat y(t)\vert \leq \mathcal E_\gamma f(\gamma u).                                                                                                                                                           
                                                                                                                                                                   \end{align*}
Notice that the reduced system output $\hat y$ also depends on $\gamma$. In Figure \ref{error_bound_num}, we see that both the error and the bound decrease in $\gamma\in [1.3, 3]$, 
where the control in (\ref{control_numerics}) with $\alpha=1$ is used. 
As anticipated before, the bound is relatively tight for 
small $\gamma$, whereas it is loosing accuracy if $\gamma$ is larger. Further numerical simulations show that the bound can perform very well if no rescaling is needed and the control energy is not too large but it is also acceptable 
in the example considered here if $\gamma$ is not much larger than $2$, see Figure \ref{error_bound_num2}. Furthermore, notice that we have not investigated the cases of $\gamma>3$ because these choices do not really improve
the error in the model reduction procedure and for too large $\gamma$ it is even increasing again.\smallskip

As already mentioned in Section \ref{sec61}, time-limited or truncated Gramians avoid a rescaling by $\gamma$ and can hence potentially lead to an improved output bound.
\begin{figure}[ht]
\begin{minipage}{0.45\linewidth}
 \hspace{-0.5cm}
\includegraphics[width=1.1\textwidth]{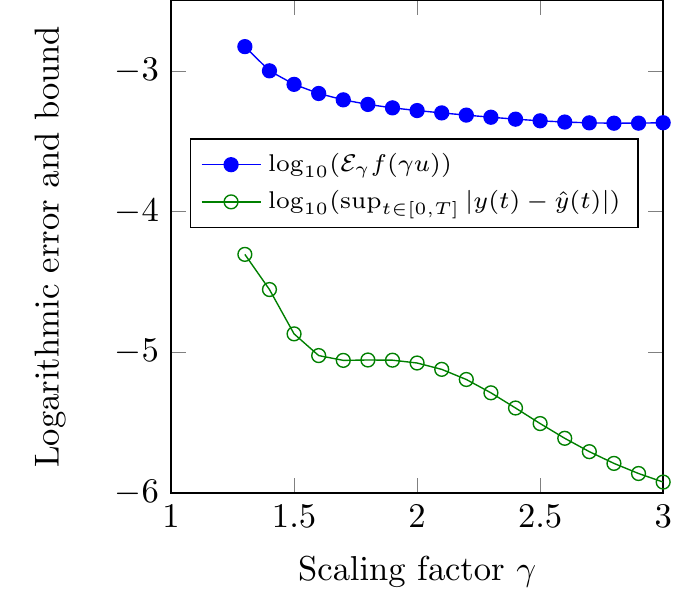}
\caption{Reduction error $\sup_{t\in [0, T]} \vert y(t) - \hat y(t)\vert$ for $r=10$, $u$ defined in (\ref{control_numerics}) and $\alpha =1$ compared with the bound 
in Corollary \ref{corh2error} for $\gamma\in [1.3, 3]$ and $T=2$.}\label{error_bound_num}
\end{minipage}\hspace{0.5cm}
\begin{minipage}{0.45\linewidth}
 \hspace{-0.5cm}
\includegraphics[width=1.1\textwidth]{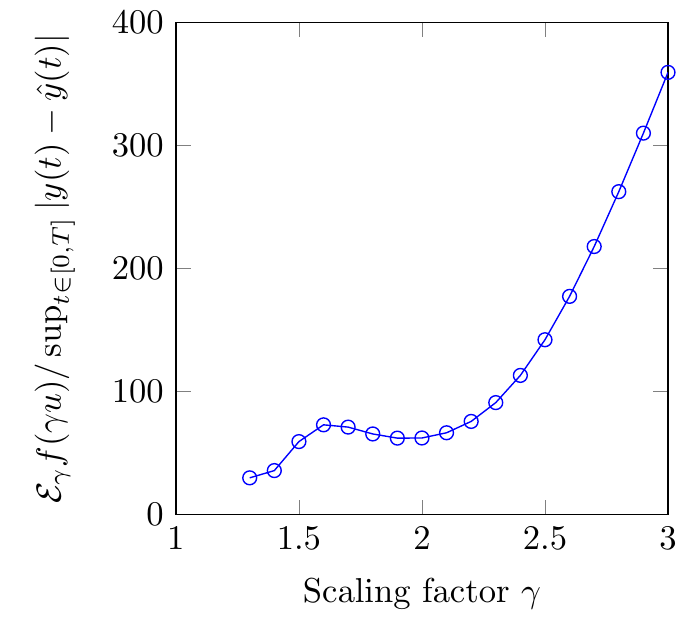}
\caption{Bound $\mathcal E_\gamma f(\gamma u)$ divided by reduction error $\sup_{t\in [0, T]} \vert y(t) - \hat y(t)\vert$ with the same parameters as in Figure \ref{error_bound_num}.}\label{error_bound_num2}
\end{minipage}
\end{figure}

\section{Conclusions}

In this paper, we studied bilinear systems in terms of asymptotic stability in the state equation and boundedness in the output.
Furthermore, we characterized reachability within bilinear equations. Moreover, we found a bound for the output errors of bilinear systems based on 
the $\mathcal H_2$-error. This error bound could finally explain why $\mathcal H_2$-optimal model order reduction techniques lead to good approximations. Such a link has been an open question for quite some time.
Subsequently, error bounds for balanced truncation and singular perturbation approximation could be derived. These bounds are the first ones for this
type of balancing related schemes considered here and they can tell us in which cases the methods perform well.

\section*{Acknowledgments}
The author thanks Tobias Damm (University of Kaiserslautern) for the fruitful discussion around the proof of Lemma \ref{lem3}. The author would also like to thank the anonymous
reviewers for their helpful, constructive and detailed comments that greatly contributed to improving the paper. 

\appendix

\section{Resolvent positive operators}\label{appendixbla}

Let $\left(H^n, \langle \cdot, \cdot\rangle_F\right)$ be the Hilbert space of symmetric $n\times n$ matrices and let $H^n_+$ be the subset of symmetric positive semidefinite matrices. We now define positive and 
resolvent positive operators on $H^n$.
\begin{defn}
A linear operator $L: H^n\rightarrow H^n$ is called positive if $L(H^n_+)\subset H^n_+$. It is resolvent positive if there is an $\alpha_0\in\mathbb R$ such that for all $\alpha>\alpha_0$ the operator $(\alpha I - L)^{-1}$ is positive.
\end{defn}\\
The operator $\mathcal L(X):=AX + X A^T$ is resolvent positive and $\Pi(X):=\sum_{k=1}^m N_k X N_k^T$ is positive for $A, N_k\in\mathbb R^{n\times n}$. This implies that the generalized Lyapunov operator $\mathcal L + \Pi$ is resolvent positive. We refer to 
\cite{damm} for a more detailed discussion and a proof. We now state an equivalent characterization for resolvent positive operators in the following. It can be found in a more general form in \cite{damm, elsner, schneidervid}.
\begin{thm}\label{equiresolpos}
A linear operator $L: H^n\rightarrow H^n$ is resolvent positive if and only if $\langle V_1, V_2\rangle_F=0$ implies $\langle L V_1, V_2\rangle_F\geq 0$ for $V_1, V_2\in H^n_+$.
\end{thm}

\section{Pending proofs}
Notice that Lemma \ref{lem3} is proved for initial time $s=0$ for simplicity of the notation. The proof is completely analogous for general initial times. 
Moreover, we write $x(t)$ instead of $x_{(s)}(t, x_0, 0)$ to shorten the notation within the proof of Lemma \ref{lem1}. 
\subsection{Proof of Lemma \ref{lem1}}\label{secprooflem1}

We apply the product rule and insert (\ref{hom_bil_eq}). Hence, we obtain
\begin{align*}
\frac{d}{dt} x(t) x^T(t) &= [\frac{d}{dt} x(t)] x^T(t) + x(t) [\frac{d}{dt} x^T(t)] \\
&= \left(Ax(t) + \sum_{k=1}^m N_k x(t) u_k(t)\right)x^T(t) + x(t)\left(x^T(t) A^T + \sum_{k=1}^m x^T(t) N_k^T u_k(t)\right)\\
&= Ax(t) x^T(t) + x(t) x^T(t) A^T + \sum_{k=1}^m\left[N_k x(t)x^T(t) u^0_k(t)+ x(t)x^T(t) N_k^T u^0_k(t)\right]\\
&= Ax(t) x^T(t) + x(t) x^T(t) A^T \\&\quad+ \sum_{k=1}^m\left[N_k x(t)x^T(t)N_k^T + x(t)x^T(t) (u^0_k(t))^2-\left(N_k - u^0_k(t)I\right) x(t)x^T(t) \left(N_k - u^0_k(t) I\right)^T\right]\\
&\leq Ax(t) x^T(t) + x(t) x^T(t) A^T + \sum_{k=1}^m\left[N_k x(t)x^T(t)N_k^T\right] + x(t)x^T(t) \left\|u^{0}(t)\right\|_2^2.
\end{align*}
This leads to the claim of this lemma.

\subsection{Proof of Lemma \ref{lem3}}\label{secprooflem3}

We set $Y:= Z-X$ and $L(Y(t)):= A Y(t) + Y(t) A^T +\sum_{k=1}^m N_k Y(t) N_k^T + Y(t) \left\|u^{0}(t)\right\|_2^2$. We subtract (\ref{matrix_ineq}) from (\ref{matrix_eq}) and obtain 
\begin{align*}
\dot Y(t) \geq L(Y(t)).
\end{align*}
We define the difference function $D(t) := \dot Y(t) - L(Y(t)) \geq 0$ and consider the following perturbed differential equation 
\begin{align*}
\dot Y_\epsilon (t) = L(Y_\epsilon(t)) + D(t) + \epsilon I
\end{align*}
with parameter $\epsilon\geq 0$ and initial state $Y_\epsilon (0)= Y(0) + \epsilon I$. We see that $Y_0(t)= Y(t)$ for all $t\geq 0$ since $Y_0 - Y$ solves the differential equation 
$\dot {\tilde Y} (t) = L(\tilde Y(t))$ with zero initial state. Since $Y_\epsilon$ continuously depends on $\epsilon$ and the initial data, we have $\lim_{\epsilon \rightarrow 0} Y_\epsilon(t) = Y_0(t) = Y(t)$ for all $t\geq 0$.
\smallskip

Let us now assume that $Y_\epsilon$ is not positive definite for $\epsilon>0$. Then, there exist a $\tilde v\ne 0$ and a $\tilde t>0$ such that $\tilde v^T Y_\epsilon(\tilde t) \tilde v \leq 0$.
We know that $f_\epsilon(v, t):= v^T Y_\epsilon(t)  v$ is positive at $t=0$ for all $v\in\mathbb R^n\setminus \{0\}$ by assumption. Since $f_\epsilon$ is non-positive in some point $(\tilde v, \tilde t)$
and due to the continuity of $t\mapsto Y_\epsilon(t)$, there is a point $t_0\in (0, \tilde t]$ for which \begin{align}\label{contradicttionprop}
v_0^T Y_\epsilon(t_0) v_0 = 0 \quad \text{and} \quad v_0^T Y_\epsilon(t) v_0 > 0,\quad t<t_0,
\end{align}
for some $v_0\ne 0$, whereas $v^T Y_\epsilon(t_0) v\geq 0$ for all other $v\in\mathbb R^n$. $L$ is a generalized Lyapunov operator and hence resolvent positive (see Appendix \ref{appendixbla}).
The identity $0 = v_0^T Y_\epsilon(t_0) v_0 = \langle Y_\epsilon(t_0), v_0v_0^T\rangle_F$, by Theorem \ref{equiresolpos}, then implies $0\leq \langle L(Y_\epsilon(t_0)), v_0v_0^T\rangle_F = v_0^T L(Y_\epsilon(t_0))v_0$.
Using these facts, we have 
\begin{align*}
\frac{d}{dt} v_0^T Y_\epsilon(t_0) v_0 = v_0^T L(Y_\epsilon(t_0)) v_0 +v_0^T D(t_0) v_0 + \epsilon \left\|v_0\right\|_2^2 > 0.
\end{align*}
Consequently, we know that there are $t<t_0$ close to $t_0$ for which $v_0^T Y_\epsilon(t) v_0<0$. This contradicts (\ref{contradicttionprop}) and hence our assumption is wrong such that 
$Y_\epsilon(t)$ is positive definite for all $t\geq 0$ and $\epsilon>0$. Taking the limit of $\epsilon \rightarrow 0$, we obtain $Y(t)\geq 0$ for all $t\geq 0$ which concludes the proof.

\bibliographystyle{plain}

\begin{thebibliography}{10}

\bibitem{typeIBT}
S.~A. Al-Baiyat and M.~Bettayeb.
\newblock {A new model reduction scheme for k--power bilinear systems}.
\newblock {\em Proceedings of the 32nd IEEE Conference on Decision and
  Control}, pages 22--27, 1993.

\bibitem{morAnt05}
A.~C. Antoulas.
\newblock {\em Approximation of Large-Scale Dynamical Systems}, volume~6 of
  {\em Adv. Des. Control}.
\newblock {SIAM} Publications, Philadelphia, PA, 2005.

\bibitem{beckerhartmann}
S.~Becker and C.~Hartmann.
\newblock {Infinite-dimensional bilinear and stochastic balanced truncation
  with explicit error bounds}.
\newblock {\em Mathematics of Control, Signals, and Systems}, 31(2):1--37,
  2019.

\bibitem{morBenB12b}
P.~Benner and T.~Breiten.
\newblock Interpolation-based $\mathcal{H}_2$-model reduction of bilinear
  control systems.
\newblock {\em {SIAM} J. Matrix Anal. Appl.}, 33(3):859--885, 2012.

\bibitem{morBenD11}
P.~Benner and T.~Damm.
\newblock Lyapunov equations, energy functionals, and model order reduction of
  bilinear and stochastic systems.
\newblock {\em {SIAM} J. Cont. Optim.}, 49(2):686--711, 2011.

\bibitem{redbendamm}
P.~Benner, T.~Damm, M.~Redmann, and Y.~R. Rodriguez~Cruz.
\newblock {Positive Operators and Stable Truncation.}
\newblock {\em {Linear Algebra and its Applications}}, 498:74--87, 2016.

\bibitem{bennerdammcruz}
P.~Benner, T.~Damm, and Y.~R. Rodriguez~Cruz.
\newblock {Dual pairs of generalized Lyapunov inequalities and balanced
  truncation of stochastic linear systems}.
\newblock {\em IEEE Trans. Autom. Contr.}, 62(2):782--791, 2017.

\bibitem{bengoygug}
P.~{Benner}, P.~{Goyal}, and S.~{Gugercin}.
\newblock {$\mathcal H_2$-Quasi-Optimal Model Order Reduction for Quadratic-Bilinear Control Systems.}
\newblock {\em {SIAM J. Matrix Anal. Appl.}}, 39(2):983--1032, 2018.

\bibitem{redmannbenner}
P.~{Benner} and M.~{Redmann}.
\newblock {Model Reduction for Stochastic Systems.}
\newblock {\em {Stoch PDE: Anal Comp}}, 3(3):291--338, 2015.

\bibitem{brunietal}
C.~Bruni, G.~DiPillo, and G.~Koch.
\newblock {On the mathematical models of bilinear systems}.
\newblock {\em Automatica}, 2(1):11--26, 1971.

\bibitem{morcondon2005}
M.~Condon and R.~Ivanov.
\newblock Nonlinear systems-algebraic gramians and model reduction.
\newblock {\em COMPEL}, 24(1):202--219, 2005.

\bibitem{damm}
T.~Damm.
\newblock {\em {Rational Matrix Equations in Stochastic Control.}}
\newblock {Lecture Notes in Control and Information Sciences 297. Berlin:
  Springer}, 2004.

\bibitem{elsner}
L.~Elsner.
\newblock {Quasimonotonie und Ungleichungen in halbgeordneten R\"aumen}.
\newblock {\em Linear Algebra and Its Applications}, 8(3):249--261, 1974.

\bibitem{flagg2015multipoint}
G.~Flagg and S.~Gugercin.
\newblock Multipoint {V}olterra series interpolation and $\mathcal{H}_2$
  optimal model reduction of bilinear systems.
\newblock {\em {SIAM} J. Matrix Anal. Appl.}, 36(2):549--579, 2015.

\bibitem{enefungray98}
W.~S. Gray and J.~Mesko.
\newblock Energy functions and algebraic {G}ramians for bilinear systems.
\newblock In {\em Preprints of the 4th IFAC Nonlinear Control Systems Design
  Symposium}, pages 103--108, Enschede, The Netherlands, 1998.

\bibitem{morGugAB08}
S.~Gugercin, A.~C. Antoulas, and C.~A. Beattie.
\newblock {$\mathcal{H}_2$} model reduction for large-scale dynamical systems.
\newblock {\em {SIAM} J. Matrix Anal. Appl.}, 30(2):609--638, 2008.

\bibitem{hartmann}
C.~{Hartmann}, B.~{Sch\"afer-Bung}, and A.~{Th\"ons-Zueva}.
\newblock {Balanced averaging of bilinear systems with applications to
  stochastic control.}
\newblock {\em {SIAM J. Control Optim.}}, 51(3):2356--2378, 2013.

\bibitem{haus}
U.~G. Haussmann.
\newblock {Asymptotic Stability of the Linear Ito Equation in Infinite
  Dimensions.}
\newblock {\em J. Math. Anal. Appl.}, 65:219--235, 1978.

\bibitem{staboriginal}
R.~Z. Khasminskii.
\newblock {Stochastic stability of differential equations.}
\newblock {Monographs and Textbooks on Mechanics of Solids and Fluids.
  Mechanics: Analysis, 7. Alphen aan den Rijn, The Netherlands; Rockville,
  Maryland, USA: Sijthoff \& Noordhoff.}, 1980.

\bibitem{Mohler}
R.~R. Mohler.
\newblock {\em {Bilinear Control Processes}}.
\newblock Academic Press, New York, 1973.

\bibitem{redstochbil}
M.~Redmann.
\newblock {Energy estimates and model order reduction for stochastic bilinear
  systems}.
\newblock {\em International Journal of Control}, 2018.

\bibitem{redmanntypeiibilinear}
M.~Redmann.
\newblock {Type II balanced truncation for deterministic bilinear control
  systems}.
\newblock {\em {SIAM J. Control Optim.}}, 56(4):2593--2612, 2018.

\bibitem{redmannspa2}
M.~Redmann.
\newblock {Type II singular perturbation approximation for linear systems with
  Lévy noise}.
\newblock {\em SIAM J. Control Optim.}, 56(3):2120--2158, 2018.

\bibitem{redmannstochbilspa}
M.~Redmann.
\newblock {A new type of singular perturbation approximation for stochastic
  bilinear systems}.
\newblock {\em Math. Control Signals Syst.}, 32(2):129--156, 2020.

\bibitem{BTtyp2EB}
M.~Redmann and P.~Benner.
\newblock {An $H_2$-Type Error Bound for Balancing-Related Model Order
  Reduction of Linear Systems with L{\'e}vy Noise}.
\newblock {\em Systems and Control Letters}, 105:1--5, 2017.

\bibitem{redSPA}
M.~Redmann and P.~Benner.
\newblock {Singular Perturbation Approximation for Linear Systems with L\'evy
  Noise}.
\newblock {\em Stochastics and Dynamics}, 18(4), 2018.

\bibitem{rugh1981nonlinear}
W.~J. Rugh.
\newblock {\em Nonlinear System Theory}.
\newblock The Johns Hopkins University Press, Baltimore, MD, 1981.

\bibitem{schneidervid}
H.~Schneider and M.~Vidyasagar.
\newblock {Cross-Positive Matrices}.
\newblock {\em SIAM Journal on Numerical Analysis}, 7(4):508--519, 1970.

\bibitem{sontag}
E.~D.~Sontag.
\newblock {Comments on integral variants of ISS}.
\newblock {\em Systems and Control Letters}, 34(1--2):93--100, 1998.

\bibitem{wonham}
W.~M.~Wonham.
\newblock On a Matrix Riccati Equation of Stochastic Control.
\newblock {\em SIAM Journal on Control}, 6(4):681--697, 1968.

\bibitem{morZhaL02}
L.~Zhang and J.~Lam.
\newblock On {$H_2$} model reduction of bilinear systems.
\newblock {\em Automatica}, 38(2):205--216, 2002.

\end{thebibliography}

\end{document}